\documentclass[11pt,a4paper,reqno]{amsart}
\usepackage[english]{babel}
\usepackage[T1]{fontenc}
\usepackage{verbatim}
\usepackage{palatino}
\usepackage{amsmath}
\usepackage{amssymb}
\usepackage{amsthm}
\usepackage{amsfonts}
\usepackage{graphicx}
\usepackage{color}
\usepackage{mathtools}

\usepackage[colorlinks = true, citecolor = black]{hyperref}
\title{The local symmetry condition in the Heisenberg group}
\author{Tuomas Orponen}
\address{Department of Mathematics and Statistics, University of Helsinki, Gustaf H\"allstr\"ominkatu 2b, 00014 University of Helsinki, Finland}
\email{tuomas.orponen@helsinki.fi}
\date{\today}
\thanks{T.O. was supported by the Academy of Finland via the Research Fellowship project \emph{Quantitative rectifiability in Euclidean and non-Euclidean spaces}, grant no. 309365.}

\newcommand{\R}{\mathbb{R}}
\newcommand{\W}{\mathbb{W}}
\newcommand{\He}{\mathbb{H}}
\newcommand{\N}{\mathbb{N}}

\newcommand{\Z}{\mathbb{Z}}

\newcommand{\calH}{\mathcal{H}}

\newcommand{\spt}{\operatorname{spt}}

\newcommand{\diam}{\operatorname{diam}}

\newcommand{\dist}{\operatorname{dist}}

\def\Barint_#1{\mathchoice
          {\mathop{\vrule width 6pt height 3 pt depth -2.5pt
                  \kern -8pt \intop}\nolimits_{#1}}%
          {\mathop{\vrule width 5pt height 3 pt depth -2.6pt
                  \kern -6pt \intop}\nolimits_{#1}}%
          {\mathop{\vrule width 5pt height 3 pt depth -2.6pt
                  \kern -6pt \intop}\nolimits_{#1}}%
          {\mathop{\vrule width 5pt height 3 pt depth -2.6pt
                  \kern -6pt \intop}\nolimits_{#1}}}

\numberwithin{equation}{section}

\theoremstyle{plain}
\newtheorem{thm}[equation]{Theorem}

\newtheorem{lemma}[equation]{Lemma}

\newtheorem{ex}[equation]{Example}
\newtheorem{cor}[equation]{Corollary}
\newtheorem{proposition}[equation]{Proposition}

\newtheorem{problem}{Problem}

\theoremstyle{definition}

\newtheorem{definition}[equation]{Definition}

\theoremstyle{remark}
\newtheorem{remark}[equation]{Remark}

\begin{document}

\begin{abstract} I propose an analogue in the first Heisenberg group $\He$ of David and Semmes' local symmetry condition (LSC). For closed $3$-regular sets $E \subset \He$, I show that the (LSC) is implied by the $L^{2}(\calH^{3}|_{E})$ boundedness of $3$-dimensional singular integrals with horizontally antisymmetric kernels, and that the (LSC) implies the weak geometric lemma for vertical $\beta$-numbers.  
\end{abstract}

\maketitle

\tableofcontents

\section{Introduction} Developing a theory of uniformly rectifiable sets in Heisenberg groups has lately attracted some attention: for results on sets of dimension $1$, see for example \cite{MR3456155,MR3512421,MR2789375,MR3678492,2018arXiv180405646C}, and for results on sets of co-dimension $1$, see \cite{CFO, CFO2, FOR, NY}. This note concerns sets of co-dimension $1$, hence dimension $3$, in the first Heisenberg group $\He$. Here are some of the basic problems motivating the research:
\begin{problem}\label{Q1} Identify natural $3$-dimensional singular integral operators in $\He$. Find necessary and sufficient conditions for $3$-regular subsets $E \subset \He$ which ensure that such operators are bounded on $L^{2}(\calH^{3}|_{E})$.
\end{problem}

In $\R^{n}$, similar questions have been studied since the $70$'s: an early result in the field is the proof by Calder\'on \cite{Calderon} and Coifman-McIntosh-Meyer \cite{CMM} that the Cauchy transform -- a $1$-dimensional singular integral operator in the plane -- is $L^{2}$-bounded on Lipschitz graphs. In higher (Euclidean) dimensions, Problem \ref{Q1} was studied extensively by David \cite{MR956767}, Semmes \cite{Semmes} and David-Semmes \cite{DS1,DS2,DS3} in the 80's and 90's, and, for example, by Tolsa \cite{MR2481953} and Nazarov-Tolsa-Volberg \cite{MR3286036} in the 2000's. This list of references is far from complete! Roughly speaking, the most natural singular integrals to consider in $\R^{n}$ are the ones with antisymmetric kernels, and they are bounded on $L^{2}(E)$ if and only if $E$ is uniformly rectifiable. This is evidently not a rigorous statement: for more precise ones, see the various characterisations of uniform rectifiability in \cite{DS1}.  

In $\He$, what are the natural analogues of singular integrals with antisymmetric kernels? The notion of antisymmetry can be directly translated to $\He$ by requiring that
\begin{equation}\label{antisymmetry} K(p^{-1}) = -K(p), \qquad p \in \He \setminus \{0\}. \end{equation}
It also seems that singular integrals with antisymmetric kernels are suitable for studying geometric problems: for example, the direct analogues in $\He^{n}$ of the $s$-dimensional Riesz kernels in $\R^{n}$ were studied by Chousionis and Mattila in \cite{Chousionis2011}. A special case of their result says that if the associated singular integral is $L^{2}$-bounded on an $s$-regular set $E \subset \He$, then $s \in \{1,2,3\}$, and $\calH^{s}|_{E}$ has some flat tangent measures. Conversely, the main result in \cite{CFO2} applies to these singular integrals with $s = 3$, and guarantees their boundedness on some smooth (intrinsic) graphs in $\He$.

There is, however, one issue with this approach: the kernel
\begin{displaymath} p \mapsto \mathcal{K}(p) := \nabla_{\He} \|p\|^{-2}, \qquad p \in \He \setminus \{0\}, \end{displaymath}
is not antisymmetric. This kernel is the horizontal gradient of the fundamental solution of the sub-Laplace equation $\bigtriangleup_{\He} u = 0$: the solutions of this equation are known as the \emph{harmonic functions in $\He$}, and they have been studied quite extensively, see \cite{BLU} and the references therein. The boundedness of the singular integral $\mathcal{R}$ with kernel $\mathcal{K}$ has consequences akin to the boundedness of the $(n - 1)$-dimensional Riesz transform in $\R^{n}$. For example, $3$-regular subsets of $\He$ on which $\mathcal{R}$ is $L^{2}$-bounded are non-removable for Lipschitz-harmonic functions in $\He$: this result is \cite[Theorem 5.1]{CFO2}, but the singular integral $\mathcal{R}$ was first applied to the removability problem by Chousionis and Mattila in \cite{CM}. Also, if one eventually hopes apply the theory of singular integrals in $\He$ to boundary value problem related to the sub-Laplace equation (following Fabes, Jodeit and Rivi\`ere \cite{MR501367} and Verchota \cite{MR769382}), or understand the behaviour of the associated (sub-)harmonic measure, then the theory needs to cover the operator $\mathcal{R}$.

So, if the kernel $\mathcal{K}$ is not antisymmetric, what is it then? It \emph{horizontally antisymmetric}, as discussed above \cite[Definition 2.5]{CFO2}:
\begin{definition}[Horizontal antisymmetry]\label{HAS} A function $\psi \colon \He \setminus \{0\} \to \R$ is called \emph{horizontally antisymmetric}, if $\psi(\bar{p}) = -\psi(p)$ for all $p \in \He \setminus \{0\}$, where $\overline{(x,y,t)} = (-x,-y,t)$. \end{definition}
This condition neither implies, or is implied, by the direct analogue of antisymmetry \eqref{antisymmetry}, but it seems to be well adapted to the geometry of $\He$. The main result, Theorem 2.10, in \cite{CFO2} shows that $3$-dimensional singular integrals in $\He$ with horizontally antisymmetric Calder\'on-Zygmund kernels are bounded on a family of intrinsic $C^{1,\alpha}$-graphs. For the precise definition of "$3$-dimensional Calder\'on-Zygmund kernel", see \cite[Section 2.2]{CFO2}. In the present paper, I also consider horizontally antisymmetric kernels, but only smooth ones (as this will make the results stronger):
\begin{definition}[Admissible kernels]\label{admissibleKernels0} A smooth function $K \colon \He \setminus \{0\} \to \R$ is called an \emph{admissible kernel}, if the following requirements are met:
\begin{itemize} 
\item $K$ is horizontally antisymmetric,
\item $K$ satisfies
\begin{displaymath} |\nabla_{\He}^{j}K(p)| \leq C(j)\|p\|^{-3 - j}, \qquad j \in \{0,1,2,\ldots\}. \end{displaymath}
\end{itemize}
\end{definition}
It is easy to check (or see the proof of \cite[Proposition 3.11(iii)]{CM}) that the admissible kernels, as above, are $3$-dimensional Calder\'on-Zygmund kernels in the sense of \cite{CFO2}. 

Here is the main result of the note:
\begin{thm}\label{main} Assume that $E \subset \He$ is closed and $3$-regular, and all singular integrals with admissible kernels are bounded on $L^{2}(\calH^{3}|_{E})$. Then $E$ satisfies the weak geometric lemma for vertical $\beta$-numbers. 
\end{thm}

Recall that an $\calH^{s}$-measurable set $E \subset \He$ is called $s$-regular, if there exists a constant $A \geq 1$ such that
\begin{displaymath} \frac{r^{s}}{A} \leq \calH^{s}(E \cap B(p,r)) \leq Ar^{s}, \qquad p \in E, \: 0 < r \leq \diam(E). \end{displaymath}

\subsection{The local symmetry condition in $\He$} The proof of Theorem \ref{main}, along with further definitions concerning singular integrals, can be found in Section \ref{SIOSection}. Satisfying the \emph{weak geometric lemma for vertical $\beta$-numbers} means that $E$ admits fairly good approximations by vertical planes at most scales and locations; the condition (without the word "vertical") was introduced by David and Semmes in \cite{DS1}, and the Heisenberg analogue was first studied in \cite{CFO}. For a precise definition, see the statement of Proposition \ref{LSImpliesWGL}.

Theorem \ref{main} is the counterpart of a result in \cite{DS1}, and follows the same chain of implications: in the terminology of \cite{DS1},
\begin{displaymath} \textup{(C1)} \quad \Longrightarrow \quad \textup{(C2)} \quad \Longrightarrow \quad \textup{(LSC)} \quad \Longrightarrow \quad \textup{(WGL)}. \end{displaymath}
Here (C1) is the main hypothesis of Theorem \ref{main}, while (C2)-(LSC) are two intermediate conditions, and (WGL) is the weak geometric lemma (for vertical $\beta$-numbers in the present context). 

The letters (LSC) stand for \emph{local symmetry condition}: this property merits a brief discussion here, because finding -- and applying -- the appropriate analogue in $\He$ is the main novelty of the note. Given two points $x,y \in \R^{n}$, the \emph{symmetric point of $y$ relative to $x$} is
\begin{displaymath} S_{x}(y) := 2x - y. \end{displaymath}
This is the point obtained by mirroring $y$ about the centre $x$. A set $E \subset \R^{n}$ will (in this note at least) be called \emph{symmetric}, if
\begin{displaymath} x,y \in E \quad \Longrightarrow \quad S_{x}(y) \in E. \end{displaymath}
More generally, a closed $m$-regular set $E \subset \R^{n}$ satisfies the (LSC), if it is "symmetric up to a small error" at most scales and places (see Definition \ref{LSC}). The argument in \cite[Section 5]{DS1} can be viewed as a quantitative proof of the following claim: a closed symmetric $m$-regular subset of $\R^{n}$ is an $m$-plane. As a corollary (of the quantitative proof), if a closed $m$-regular set satisfies the (LSC), then it is "close" to an $m$-plane at most scales and locations. This is the proof idea of the implication (LSC) $\Longrightarrow$ (WGL) in $\R^{n}$.

What should be the analogue of a symmetric point in $\He$? Imitating the Euclidean definition, one could set $S^{\He}_{p}(q) = p \cdot q^{-1} \cdot p$ for $p,q \in \He$. Then, one could define the (LSC), and one might even be able to prove the implication (LSC) $\Longrightarrow$ (WGL). However, the resulting notion of (local) symmetry would have nothing to do with horizontally antisymmetric kernels. 

It seems that a more appropriate definition of \emph{symmetric point} is the following:
\begin{equation}\label{symmetricPoint} \Sigma_{p}(q) := p \cdot \overline{p^{-1} \cdot q}, \qquad p,q \in \He, \end{equation}
where $\overline{(x,y,t)} = (-x,-y,t)$, as in Definition \ref{HAS}. It turns out that this notion of symmetric points can be characterised in various different ways, as will be discussed in Section \ref{characterisations}. Also, Theorem \ref{sym3reg} below shows that closed $3$-regular symmetric sets in $\He$ are subsets of vertical planes (they can be strict subsets of planes in $\He$, as opposed to $\R^{n}$). Moreover, Theorem \ref{C1ImpliesWGL} shows that the (LSC) derived from \eqref{symmetricPoint} is implied by the $L^{2}$-boundedness of singular integrals with admissible kernels, as in Definition \ref{admissibleKernels0}. These observations combined (essentially) prove Theorem \ref{main}.

\subsection{Basic notation and other conventions} In the first Heisenberg group $\He = (\R^{3},\cdot)$, I will use the group law
\begin{displaymath} (x,y,t) \cdot (x',y',t') = (x + x',y + y',t + t' + \tfrac{1}{2}(xy' - yx')), \end{displaymath} 
and the (Kor\'anyi) metric $d(p,q) = \|q^{-1} \cdot p\|$ induced by the norm-like quantity
\begin{displaymath} \|p\| := ((x^{2} + y^{2})^{2} + 16t^{2})^{1/4}, \qquad p = (x,y,t) \in \He. \end{displaymath}
For $r > 0$, I write $\delta_{r} \colon \He \to \He$ for the dilatation
\begin{displaymath} \delta_{r}(p) = (rx,ry,r^{2}t), \qquad p = (x,y,t) \in \He. \end{displaymath}
Open balls in the metric $d$ will be denoted by $B(p,r)$, with $p \in \He$ and $r > 0$. The $s$-dimensional Hausdorff measure on $\He$, defined via the metric $d$, is denoted by $\calH^{s}$. The notation $A \lesssim_{P} B$ means that $A \leq CB$, where $C \geq 1$ is a constant depending only on the parameter $P$. The two-sided inequality $A \lesssim_{P} B \lesssim_{Q} \lesssim A$ is abbreviated to $A \sim_{P,Q} B$.

\section{The structure of symmetric sets in $\He$}\label{section2}

\subsection{Symmetric points in $\He$}\label{characterisations} Here is the Euclidean definition once more:

\begin{definition}[Symmetric points in $\R^{2}$] Given two points $x,y \in \R^{2}$, we denote by $S_{x}(y) = 2x - y$ the \emph{symmetric point of $y$ relative to $x$}. A set $E \subset \R^{2}$ is \emph{symmetric}, if $S_{x}(y) \in E$ for all $x,y \in E$. \end{definition}

I denote by $\pi \colon \He \to \R^{2}$ the projection $\pi(x,y,t) = (x,y)$, which is a group homomorphism $(\He,\cdot) \to (\R^{2},+)$. Points in $\R^{2}$ are typically denoted by $z$, and points in $\He$ are typically denoted by $p,q$. I will often identify $\R^{2}$ with the horizontal plane $H := \{(x,y,0) : x,y \in \R\} \subset \He$. In particular, if $z \in \R^{2}$ and $p \in \He$, the notation $p \cdot z$ stands for $p \cdot (z,0)$. Also, without special mention, I often write points $p \in \He$ in the form $p = (z,t)$, where $z = \pi(p)$. 

\begin{definition}[Lifts] Given two points $p,q \in \He$, the \emph{$p$-lift} of $q$ is the point
\begin{equation}\label{liftDef} q|_{p} = p\cdot (\pi(q) - \pi(p)). \end{equation} 
Then $q|_{p}$ is the unique point in the plane $p \cdot H$ which $\pi$-projects to $\pi(q)$: in other words, $q|_{p}$ is characterised by the properties
\begin{equation}\label{lift} \pi(q|_{p}) = \pi(q) \quad \text{and} \quad p^{-1} \cdot q|_{p} \in H. \end{equation}
Given a sequence of points $(q_{1},\ldots,q_{n}) \in \He^{n}$, and a (base) point $p \in \He$, the $p$-lift of the sequence $\sigma = (q_{1},\ldots,q_{n})$ is the sequence $\sigma|_{p} = (p_{1},\ldots,p_{n}) \in \He^{n}$ determined by $p_{1} = q_{1}|_{p}$ and
\begin{displaymath} p_{j + 1} = q_{j + 1}|_{p_{j}} \qquad 1 \leq j \leq n - 1. \end{displaymath} 
\end{definition}
In fact, I will only need to lift points and sequences lying on $\R^{2}$. 

\begin{definition}[Symmetric points in $\He$] Given two points $p,q \in \He$, the \emph{symmetric point of $q$ relative to $p$} is the point
\begin{displaymath} \Sigma_{p}(q) := S_{\pi(p)}(\pi(q))|_{q} = q \cdot [S_{\pi(p)}(\pi(q)) - \pi(q)]. \end{displaymath}
A set $E \subset \He$ is \emph{symmetric}, if $\Sigma_{p}(q) \in E$ for all $p,q \in E$.
\end{definition}

The symmetric point of $q$ relative to $p$ is obtained by projecting both points to the plane $H \cong \R^{2}$, then mirrorig $\pi(q)$ relative to the centre $\pi(p)$, and finally lifting the result back to the horizontal plane $q \cdot H$. For computational purposes, I record a simple formula for $\Sigma_{p}(q)$, which also appeared in \eqref{symmetricPoint}. For $p = (z,t) \in \He$, write $\bar{p} := (-z,t)$. For later, I note that $p \mapsto \bar{p}$ is clearly a group isomorphism and an isometry. 

\begin{lemma}\label{lemma1} For $p,q \in \He$,
\begin{displaymath} \Sigma_{p}(q) = p \cdot \overline{p^{-1} \cdot q}. \end{displaymath}
\end{lemma}

\begin{proof} Write $w := p \cdot \overline{p^{-1} \cdot q}$. Then, by \eqref{lift}, one has $w = S_{\pi(p)}(\pi(q))|_{q} =: \Sigma_{p}(q)$, if and only if
\begin{itemize}
\item[(i)] $\pi(w) = S_{\pi(p)}(\pi(q))$, and
\item[(ii)] $q^{-1} \cdot w \in H$.
\end{itemize}
To verify (i), note that $\pi(\overline{w}) = -\pi(w)$ for all $w \in \He$. Then compute as follows:
\begin{displaymath} \pi(w) = \pi(p) - \pi(p^{-1} \cdot q) = \pi(p) + (\pi(p) - \pi(q)) =: S_{\pi(p)}(\pi(q)).  \end{displaymath}
To see (ii), note that $q^{-1} \cdot w = (q^{-1} \cdot p) \cdot (\overline{p^{-1} \cdot q}) = v^{-1} \cdot \bar{v}$, where $v := (z,t) := p^{-1} \cdot q$. Then, 
\begin{displaymath} v^{-1} \cdot \bar{v} = (-z,-t) \cdot (-z,t) = (-2z,0) \in H, \end{displaymath}
as claimed. The proof is complete. \end{proof}
\begin{ex} From Lemma \ref{lemma1}, one immediately gets
\begin{displaymath} \Sigma_{0}(q) = 0 \cdot \overline{0^{-1} \cdot q} = \bar{q}. \end{displaymath}
This equation is, later on, the link to horizontally antisymmetric kernels. 
\end{ex}

If $p = (a,b,c)$ and $q = (x,y,t)$, one can further compute that
\begin{equation}\label{form10} \Sigma_{p}(q) = p \cdot \overline{p^{-1} \cdot q} = (2a - x, 2b - y,t - ax + by). \end{equation}

\begin{remark} Here is one more description of $\Sigma_{p}(w)$ (which I will not explicitly need). If $p^{-1} \cdot q$ lies on the $t$-axis, then $\Sigma_{p}(q)$ is simply $q$. Otherwise, there is a unique vertical subgroup $\W$ satisfying $p \cdot \W = q \cdot \W$, and further $q$ lies on a unique horizontal line $L \subset p \cdot \W$. The line $L$ contains a unique point $w \neq q$ with 
\begin{displaymath} d(p,w) = d(p,q), \end{displaymath}
and this point is $w = \Sigma_{p}(q)$. This is easy to verify: it is immediate from the definition that $\Sigma_{p}(q)$ lies on the same vertical plane $p \cdot \W$ as $p,q$ (since $S_{\pi(p)}(\pi(q))$ does), and also on some common horizontal line with $q$ -- namely the unique such line contained in $p \cdot \W$. Finally, from Lemma \ref{lemma1} one infers that
\begin{displaymath} d(p,\Sigma_{p}(q)) = \|p^{-1} \cdot p \cdot \overline{p^{-1} \cdot q}\| = \|\overline{p^{-1} \cdot q}\| = \|p^{-1} \cdot q\| = d(p,q). \end{displaymath}
\end{remark}

\subsection{Properties of symmetric sets in $\He$} The next aim is to describe the structure of symmetric sets in $\He$: it turns out that they are either subsets of vertical planes, or then something substantially larger. First, a nearly trivial lemma:
\begin{lemma} Assume that $E \subset \He$ is symmetric. Then $\pi(E) \subset \R^{2}$ is symmetric. \end{lemma}
\begin{proof} Let $\pi(p),\pi(q) \in \pi(E)$ with $p,q \in E$. Then $\Sigma_{p}(q) \in E$, hence
\begin{displaymath} S_{\pi(p)}(\pi(q)) = \pi(\Sigma_{p}(q)) \in \pi(E). \end{displaymath}
This means, by definition, that $\pi(E)$ is symmetric in $\R^{2}$.
\end{proof}

Motivated by this observation, one is first tempted to study symmetric sets in $\R^{2}$:
\begin{lemma}\label{symmetricR2} Assume that $A \subset \R^{2}$ is symmetric, and $0,a,b \in A$. Then
\begin{displaymath} (\Z a + \Z b) \setminus [(2\Z + 1)a + (2\Z + 1)b] \subset A. \end{displaymath}
\end{lemma}

\begin{proof} We will show this in the case $a = e_{1} = (1,0)$ and $b = e_{2} = (0,1)$. By induction, it suffices to show that every point
\begin{displaymath} (m,n) \in [\Z \times \Z \setminus (2\Z + 1) \times (2\Z + 1)] \setminus \{0,e_{1},e_{2}\} \end{displaymath}
can be expressed as $(m,n) = S_{(m_{1},n_{1})}(m_{2},n_{2})$, where
\begin{displaymath} (m_{1},n_{1}), (m_{2},n_{2}) \in  \Z \times \Z \setminus (2\Z + 1) \times (2\Z + 1), \end{displaymath}
and 
\begin{equation}\label{form14} |m_{1}| + |n_{1}| + |m_{2}| + |n_{2}| < 2(|m| + |n|). \end{equation}
The proof is most clearly conveyed by a picture, see Figure \ref{fig1}: use points close to the origin to construct further points by "jumping over" the previously constructed points. It is slightly curious that points $(m,n) \in (2\Z + 1) \times (2\Z + 1)$ do not appear.
\begin{figure}[h!]
\begin{center}
\includegraphics[scale = 0.9]{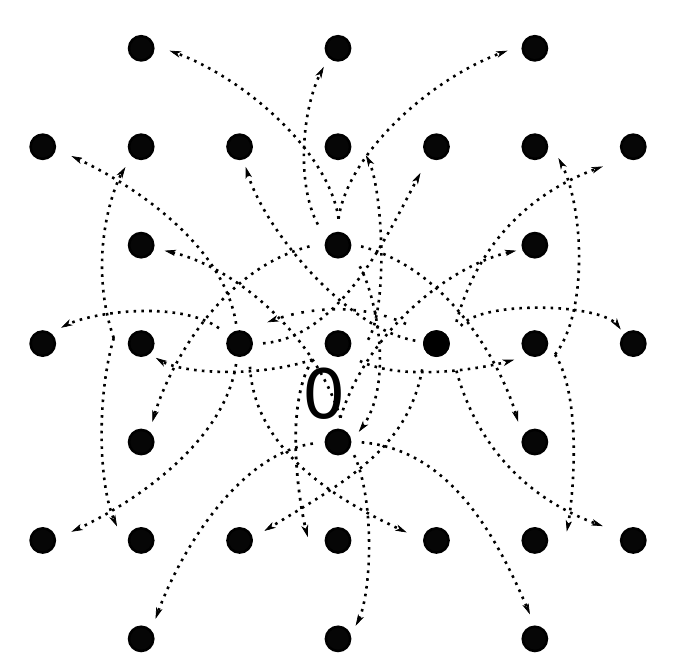}
\caption{Constructing the set $\Z \times \Z \setminus (2\Z + 1) \times (2\Z + 1)$ from the generators $\{(0,0),(1,0),(0,1)\}$. An arrow ending at a point shows how it can be found as a symmetric point of two previously constructed points.}\label{fig1}
\end{center}
\end{figure}
The rigorous proof requires plenty of case chase depending on whether $m$ or $n$ is even, and whether $m \leq 0$, $m > 0$, $n \leq 0$, $n > 0$. I only consider the case when $m,n \geq 0$ and $m$ is even (the case where $n$ is even is symmetric). Also, the case where either $m = 0$ or $n = 0$ is easy (for example $(m,0) = S_{(m - 1,0)}(m - 2,0)$), so I assume that $m,n \geq 1$. If also $n$ is even, I note that 
\begin{displaymath} (m - 2,n),(m - 1,n) \in \Z \times \Z \setminus (2\Z + 1) \times (2\Z + 1). \end{displaymath}
Clearly \eqref{form14} is also satisfied, and
\begin{displaymath} S_{(m - 1,n)}(m - 2,n) = (2(m - 1) - (m - 2),2n - n) = (m,n), \end{displaymath}
as required. This situation corresponds to (roughly) vertical and horizontal arrows in Figure \ref{fig1}. If $n$ is odd, then consider the points 
\begin{displaymath} (m - 2,n - 2),(m - 1,n - 1) \in \Z \times \Z \setminus (2\Z + 1) \times (2\Z + 1), \end{displaymath}
and note that \eqref{form14} is again satisfied. Further,
\begin{displaymath} S_{(m - 1,n - 1)}(m - 2,n - 2) = (2(m - 1) - (m - 2),2(n - 1) - (n - 2)) = (m,n), \end{displaymath}
as desired. This corresponds to the diagonal arrows in Figure \ref{fig1}. We leave the other cases to the reader, as it is quite easy to see from Figure \ref{fig1} what to do. The case of general $a,b$ reduces to the one treated above, as the map $\iota \colon (m,n) \mapsto am + bn$ commutes with $S$:
\begin{displaymath} S_{\iota(m_{1},n_{1})}(\iota(m_{2},n_{2})) = \iota (S_{(m_{1},n_{1})}(m_{2},n_{2})). \end{displaymath}
In other words, if $am + bn \in (\Z a + \Z b) \setminus [(2\Z + 1)a + (2\Z + 1)b]$, then then one can first use the argument above to find $(m_{1},n_{1}),(m_{2},n_{2})$ with $S_{(m_{1},n_{2})}(m_{2},n_{2}) = (m,n)$, and then
\begin{displaymath} S_{am_{1} + bn_{1}}(am_{2} + bn_{2}) = \iota(S_{(m_{1},n_{2})}(m_{2},n_{2})) = \iota(m,n) = am + bn. \end{displaymath}
This completes the proof. \end{proof}

\begin{remark} The assumption that $0 \in A$ is not essential: if $a,b,c \in A$, then
\begin{displaymath} a + (\Z(b - a) + \Z(c - a)) \setminus [(2\Z + 1)(b - a) + (2\Z + 1)(c - a)] \subset A, \end{displaymath}
which follows from the previous lemma applied to the (symmetric) set $A - a$.
\end{remark}

Next, I aim to show that if $E \subset \He$ is symmetric, then the (symmetric) set $\pi(E)$ contains plenty of sequences, whose lifts lie inside $E$. 
\begin{definition} A \emph{checkers sequence} in a planar set $A \subset \He$ is a sequence $(x_{1},\ldots,x_{n}) \subset A^{n}$ with the following property: for every $1 \leq j \leq n - 1$, there exists a point $y_{j} \in A$ such that $x_{j + 1} = S_{y_{j}}(x_{j})$.
\end{definition}

Thus, checkers sequences are those sequences in $A$ which can be obtained by "jumping over other points in $A$". The point of the definition is that if $E$ is symmetric, then checkers sequences in $\pi(E)$ can be lifted without leaving $E$:
\begin{lemma}\label{liftingInE} Let $E \subset \He$ be symmetric, and let $q \in E$. Let $(z_{1},\ldots,z_{n}) \in (\pi(E))^{n}$ be a checkers sequence with $z_{1} = \pi(q)$. Then the $q$-lift of $(z_{1},\ldots,z_{n})$ is contained in $E$.
\end{lemma}

\begin{proof} Let $q_{1},\ldots,q_{n} \subset \He$ be the $q$-lift of $z_{1},\ldots,z_{n}$, so that in particular $\pi(q_{j}) = z_{j}$ for all $1 \leq j \leq n$. Then
\begin{displaymath} q_{1} = z_{1}|_{q} = \pi(q)_{q} = q \in E. \end{displaymath}
Next, assume inductively that $q_{j} \in E$ for some $j \geq 1$: the aim is to prove that $q_{j + 1} \in E$. By definition of the $q$-lift of a sequence, one has
\begin{equation}\label{form11} q_{j + 1} = z_{j + 1}|_{q_{j}}, \end{equation}
where $q_{j} \in E$. Since $(z_{1},\ldots,z_{n})$ is a checkers sequence, there exists $\pi(p_{j}) \in \pi(E)$ such that $z_{j + 1} = S_{\pi(p_{j})}(z_{j}) = S_{\pi(p_{j})}(\pi(q_{j}))$. Combining this with \eqref{form11} and the symmetry assumption (on $E$) yields
\begin{displaymath} q_{j + 1} = S_{\pi(p_{j})}(\pi(q_{j}))|_{q_{j}} =: \Sigma_{p_{j}}(q_{j}) \in E. \end{displaymath} 
This completes the induction. \end{proof}

So, to construct more points from existing points in a symmetric set $E \subset \He$, it might suffice to find checkers sequences in $\pi(E)$, and lift them. This leads to the question: which sequences in $\pi(E)$ are checkers sequences?
\begin{lemma} Assume that $A \subset \R^{2}$ is symmetric, and $0,a,b \in A$. A sequence $\sigma = (z_{1},\ldots,z_{n}) \in (2\Z a + 2\Z b)^{n} \subset A^{n}$ is called \emph{connected}, if 
\begin{displaymath} z_{j + 1} \in \{z_{j} + 2a, z_{j} - 2a,z_{j} + 2b,z_{j} - 2b\}, \qquad 1 \leq j \leq n - 1. \end{displaymath}
Connected sequences are checkers sequences in $A$.
\end{lemma}
\begin{proof} Let $\sigma = (z_{1},\ldots,z_{n}) \in (2\Z a + 2\Z b)^{n}$ be a connected sequence: fix $1 \leq j \leq n - 1$, and assume for example that $z_{j + 1} = z_{j} - 2a$. Since $z_{j} \in 2\Z a + 2\Z b$, we see that
\begin{displaymath} y_{j} := z_{j} - a \in (\Z a + \Z b) \setminus [(2\Z + 1)a + (2\Z + 1)b] \subset A \end{displaymath}
by Lemma \ref{symmetricR2}. Further,
\begin{displaymath} z_{j + 1} = z_{j} - 2a = S_{z_{j} - a}(z_{j}) = S_{y_{j}}(z_{j}), \end{displaymath}
which means that $\sigma$ is a checkers sequence in $A$. The other cases are similar.
\end{proof}

The main consequence of the lemma is the existence of closed checkers sequences:

\begin{lemma}\label{symmetriLoops} Let $A \subset \R^{2}$ be symmetric with $0,a,b \in A$. Then, for any $z \in (2\Z a + 2\Z b)$, the sequences
\begin{displaymath} \sigma_{z}^{+} := (z,z + 2a, z + 2a + 2b,z + 2b,z) \quad \text{and} \quad \sigma_{z}^{-} := (z,z - 2a,z - 2a - 2b,z - 2b,z) \end{displaymath}
are checkers sequences (that is, checkers loops) in $A$.
\end{lemma}
\begin{proof} The loops $\sigma_{z}^{+},\sigma_{z}^{-} \in (2\Z a + 2\Z b)^{5}$ are evidently connected, so the claim follows from the previous lemma. \end{proof}

I suppress the dependence of $\sigma_{z}^{+}$ and $\sigma_{z}^{-}$ on $a,b$ from the notation, because the points $a,b$ will be "fixed" in future applications. It is one of the most fundamental features of $\He$ that the $q$-lift of a loop ends up either strictly "above" or "below" $q$. Here are the numbers:
\begin{lemma} Let $z,a = (a_{1},a_{2}),b = (b_{1},b_{2}) \in \R^{2}$, and let $p = (z,t) \in \He$. Then, the $p$-lift of the loop $\sigma_{z}^{+}$ terminates at $(z,t + 4\det(a,b))$, and the $p$-lift of the loop $\sigma^{-}_{z}$ terminates at $(z,t-4\det(a,b))$, where $\det(a,b) := a_{1}b_{2} - a_{2}b_{1}$.
\end{lemma} 

\begin{proof} Write $\sigma^{+} := (z_{1},z_{2},z_{3},z_{4},z_{5})$ for brevity, where 
\begin{displaymath} z_{1} = z = z_{5}, \quad z_{2} = z + 2a, \quad z_{3} = z+ 2a + 2b, \quad \text{and} \quad z_{4} = z + 2b. \end{displaymath}
By \eqref{liftDef}, the $p$-lift of the loop $\sigma^{+}$ is the sequence consisting of the points
\begin{displaymath} p, \quad p \cdot [z_{2} - z_{1}], \quad p \cdot [z_{2} - z_{1}] \cdot [z_{3} - z_{2}], \quad p \cdot [z_{2} - z_{1}] \cdot [z_{3} - z_{2}] \cdot [z_{4} - z_{3}], \end{displaymath} 
and 
\begin{displaymath} p \cdot [z_{2} - z_{1}] \cdot [z_{3} - z_{2}] \cdot [z_{4} - z_{3}] \cdot [z_{5} - z_{4}] = p \cdot 2a \cdot 2b \cdot (-2a) \cdot (-2b).  \end{displaymath}
So, the lift terminates at $p \cdot 2a \cdot 2b \cdot (-2a) \cdot (-2b)$, and one may easily compute that
\begin{displaymath} 2a \cdot 2b \cdot (-2a) \cdot (-2b) = (0,0,2a_{1}2b_{2} - 2a_{2}2b_{1}) = (0,0,4\det(a,b)), \end{displaymath} 
so $p \cdot 2a \cdot 2b \cdot (-2a) \cdot (-2b) = (z,t + 4\det(a,b))$. The proof is similar for $\sigma^{-}$. \end{proof}

Now, we are prepared to study the structure of symmetric sets in $\He$:
\begin{proposition}\label{discreteStructure} Let $E \subset \He$ be symmetric, and assume that $0,a,b \in \pi(E)$. Then
\begin{displaymath} 2\Z a + 2\Z b \subset \pi(E), \end{displaymath}
and and for all $z \in 2\Z a \times 2\Z b$, there exists a number $t_{z} \in \R$ such that
\begin{displaymath} (z,t_{z} + 4\Z \det(a,b)) \subset E. \end{displaymath}
\end{proposition}
\begin{proof} Since $\pi(E)$ is symmetric, the first claim follows immediately from Lemma \ref{symmetricR2}. Consequently, for $z \in 2\Z a + 2\Z b \subset \pi(E)$, there exists $t_{z} \in \R$ such that $p = (z,t_{z}) \in E$. The sequences $\sigma^{+}_{z}$ and $\sigma^{-}_{z}$ are checkers loops, so their $p$-lifts are contained in $E$ by Lemma \ref{liftingInE}. In particular, their endpoints are contained in $E$, and these points are
\begin{equation}\label{form15} \{p_{1},p_{2}\} := (z,t_{z} \pm 4\det(a,b)) \end{equation}
by the previous lemma. Now, the same argument can be iterated, replacing $p$ by the two points in \eqref{form15} (that is, considering the $p_{1}$- and $p_{2}$-lifts of $\sigma^{\pm}_{z}$, which again terminate in the set $(z, t_{z} + 4\Z\det(a,b)) \cap E$). This proves the second statement.  \end{proof}

\subsection{Structure of symmetric $3$-regular sets in $\He$} 
\begin{thm}\label{sym3reg} Let $E \subset \He$ be symmetric and $3$-regular. Then $E$ is contained on a vertical plane.
\end{thm} 

\begin{proof} Assume to the contrary: $E$ is not contained on a vertical plane, which means that $\pi(E) \subset \R^{2}$ is not contained on a line. After a translation, one may assume that $0 \in \pi(E)$. Further, there exist two linearly independent vectors $a,b \in \pi(E) \setminus \{0\}$. From Proposition \ref{discreteStructure}, one infers that $E$ contains all points of the form
\begin{equation}\label{form16} (z,t_{z} + 4k \det(a,b)), \qquad (z,k) \in (2\Z a \times 2\Z b) \times \Z, \end{equation}
form some $t_{z} \in \R$. Now, to reach a contradiction, it remains to count how many of these points there are in $B(0,M)$, for a suitably large $M \geq 1$, and find that this contradicts the $3$-regularity of $E$. Note that $B(0,M)$ contains a box of the form $R := [-L,L] \times [-L,L] \times [-L^{2},L^{2}]$, where $L \sim M$. Evidently, $R$ contains $\sim_{a,b} M^{4}$ points of the form \eqref{form16}, where the implicit constant depends on $a,b$, and $\det(a,b) \neq 0$. Also, all points of the form \eqref{form16} are pairwise separated by $\rho \sim_{a,b} 1$. Since 
\begin{displaymath} \calH^{4}(E \cap B(p,\rho)) \gtrsim \rho^{3}, \qquad p \in E, \end{displaymath}
by the $3$-regularity of $E$, one concludes that
\begin{displaymath} M^{4}\rho^{3} \lesssim_{a,b} \calH^{3}(E \cap B(0,M)) \lesssim M^{3}. \end{displaymath} 
This gives the desired contradiction for $M \gg \rho^{-3} \sim_{a,b} 1$. \end{proof}

\begin{remark} Any union of horizontal lines contained in a fixed vertical plane is a symmetric set. Such unions can easily be $3$-regular without covering all of the plane. So, Theorem \ref{sym3reg} cannot be upgraded to the statement that closed symmetric $3$-regular sets are vertical planes. \end{remark}

\section{The local symmetry condition and the weak geometric lemma} 
The following notion is a relaxed and localised version of symmetry. It is a slightly weaker variant of a condition appearing in \cite[Section 4]{DS1}:
\begin{definition}\label{symmetry} Fix $\tau > 0$. A closed set $E \subset \He$ is called \emph{$\tau$-symmetric} in a ball $B(p,r)$, if for all $q_{1},q_{2} \in E \cap B(p,r)$, the exists $q_{1}' \in E \cap B(q_{1},r\tau)$ such that $\dist(\Sigma_{q_{1}'}(q_{2}),E) \leq \tau r$. \end{definition}
\begin{remark}\label{DScomparison} The original David-Semmes definition given at the head of \cite[Section 4]{DS1} is slightly simpler: the direct analogue here would say that $E$ is $\tau$-symmetric in $B(p,r)$, if for all $q_{1},q_{2} \in E \cap B(p,r)$ one has $\dist(\Sigma_{q_{1}}(q_{2}),E) \leq \tau r$. The technical problem behind the additional twist in Definition \ref{symmetry} is that the map $q_{1} \mapsto \Sigma_{q_{1}}(q_{2})$ is not Lipschitz in $\He$. This is easiest to observe when $q_{2} = 0$, because $q_{1} \mapsto \Sigma_{q_{1}}(0)$ is essentially the projection to the $xy$-plane: by Lemma \ref{lemma1},
\begin{displaymath} \Sigma_{(x,y,t)}(0) = (x,y,t) \cdot \overline{(x,y,t)^{-1} \cdot 0} = (x,y,t) \cdot (x,y,-t) = (2x,2y,0) =: 2\pi(x,y,t). \end{displaymath}
Now consider points of the form $p = (x,0,0)$ and $q = (x,y,-xy/2)$ with $|x| \gg |y|$:
\begin{displaymath} d(\pi(p),\pi(q)) = d((x,0,0),(x,y,0)) = \|(0,-y,\tfrac{yx}{2})\| \gtrsim \sqrt{|yx|}, \end{displaymath}
yet
\begin{displaymath} d(p,q) = \|(-x,-y,-xy/2) \cdot (x,0,0)\| = \|(0,-y,-\tfrac{xy}{2} + \tfrac{xy}{2})\| \sim |y|, \end{displaymath} 
so $d(\pi(p),\pi(q)) \gg d(p,q)$, and hence $q_{1} \mapsto \Sigma_{q_{1}}(0)$ is not Lipschitz. So, in $\He$, the direct analogue of the David-Semmes definition would be quite unstable. To elaborate a little more, the next section contains an argument showing that non-$\tau$-symmetric balls $B(p,r)$ are rare: if one used the David-Semmes definition, such balls would, by definition, contain a pair of points $q_{1},q_{2}$ with $\dist(\Sigma_{q_{1}}(q_{2}),E) > \tau r$. But the argument in the next section needs more, namely that 
\begin{equation}\label{selfImprovement} \dist(\Sigma_{q'_{1}}(q'_{2}),E) \geq \tau r/2, \quad \text{whenever } d(q_{1}',q_{1}) \ll \tau r \text{ and } d(q_{2}',q_{2}) \ll \tau r. \end{equation}
Such a "self-improvement" of non-symmetry is automatic if both maps $q_{1},q_{1} \mapsto \Sigma_{q_{1}}(q_{2})$ are Lipschitz -- and, conversely, seems impossible to deduce in the present setting. On the other hand, if $B(p,r)$ is non-$\tau$-symmetric in the sense of Definition \ref{symmetry}, then \eqref{selfImprovement} is easily seen to be true (also using that $q_{2} \mapsto \Sigma_{q_{1}}(q_{2})$ is $1$-Lipschitz). \end{remark}

From now on, I only discuss closed $3$-regular sets $E \subset \He$. 
\begin{definition}[Local symmetry condition]\label{LSC} A closed $3$-regular set $E$ satisfies the \emph{local symmetry condition}, if the non-$\tau$-symmetric balls centred on $E$ satisfy a Carleson packing condition, for any $\tau > 0$. More precisely, for every $\tau > 0$ there is a constant $C_{\tau} > 0$ such that the following holds: for all $p_{0} \in E$ and $R > 0$, 
\begin{equation}\label{LS} \int_{0}^{R} \calH^{3}(\{p \in B(p_{0},R) : E \text{ is not $\tau$-symmetric in } B(p,r)\}) \, \frac{dr}{r} \leq C_{\tau}R^{3}. \end{equation}
\end{definition}

The main purpose of this section is to demonstrate that the local symmetry condition implies the weak geometric lemma for vertical $\beta$-numbers. Recall (from \cite[Definition 3.3]{CFO} for example) that the vertical $\beta$-number of $E$ in a ball $B(p,r)$ is the quantity
\begin{displaymath} \beta_{E}(p,r) := \inf_{\W} \sup_{y \in B(p,r) \cap E} \frac{d(y,\W)}{r}, \end{displaymath}
where the $\inf$ runs over all vertical planes $\W$ (that is, translates of planes containing the $t$-axis). The \emph{weak geometric lemma for vertical $\beta$-numbers}, introduced in \cite[Definition 3.5]{CFO}, is the statement \eqref{WGL} below: in short, balls centred on $E$ with non-negligible vertical $\beta$-numbers satisfy a Carleson packing condition.
\begin{proposition}\label{LSImpliesWGL} Assume that $E \subset \He$ is a closed $3$-regular set satisfying the \textup{(LSC)}. Then, for any $\epsilon > 0$, the estimate
\begin{equation}\label{WGL} \int_{0}^{R} \calH^{3}(\{p \in B(p_{0},R) : \beta_{E}(p,r) \geq \epsilon\}) \, \frac{dr}{r} \leq CR^{3} \end{equation} 
holds for all $p_{0} \in E$ and $R > 0$. The constant $C \geq 1$ depends on $\epsilon$, the constants in the \textup{(LSC)}, and the $3$-regularity constant of $E$.
\end{proposition}

The weak geometric lemma in $\R^{n}$ was originally introduced by David and Semmes, see \cite[Section 5]{DS1}. Theorem \ref{LSImpliesWGL} easily follows from the next lemma, which states that if $\beta_{E}(p,r) \geq \epsilon$, then there is a constant $\tau = \tau(\epsilon) > 0$ and a ball $B$ "comparable" to $B(p,r)$ such that $E$ is not $\tau$-symmetric in $B$. 

\begin{lemma}\label{betasAndSymmetry} For every $A,\epsilon > 0$ there is a constant $\tau = \tau(A,\epsilon) > 0$ such that the following holds. Assume that $E \subset \He$ is closed and $3$-regular with constant at most $A$, and $B(p,r)$ is a ball with $p \in E$ and $r > 0$ such that $\beta_{E}(p,r) \geq \epsilon$. Then, there exist two points $q_{1},q_{2} \in E \cap B(p,r/\tau)$ such that 
\begin{displaymath} \dist(\Sigma_{q_{1}'}(q_{2}),E) > \tau r \quad \text{ for all } q_{1}' \in E \cap B(q_{1},\tau r). \end{displaymath}
In particular, $E$ is not $\tau^{2}$-symmetric in $B(p,r/\tau)$.
\end{lemma}

\begin{proof} If the conclusion fails for sufficiently small $\tau > 0$, then $E$ is "essentially symmetric" in $B(p,r/\tau)$ and hence $E \cap B(p,r)$ should be contained in an arbitrarily small neighbourhood of a vertical plane, violating $\beta_{E}(p,r) \geq \epsilon$. This could be done so -- emulating arguments from the previous section -- that the dependence between $A,\epsilon$ and $\tau$ becomes effective; this approach seems so tedious, however, that I resort to compactness.

In other words, I make a counter assumption: for every $\tau = 1/i$, there exists a closed $3$-regular set $E_{i}$, with regularity constant at most $A$, and with the following properties:
\begin{itemize}
\item $0 \in E_{i}$ and $\beta_{E_{i}}(0,1) \geq \epsilon$,
\item For all $p,q \in E_{i} \cap B(0,i)$ there exists $p' \in E \cap B(p,1/i)$ with
\begin{displaymath} \dist(\Sigma_{p'}(q),E_{i}) \leq \frac{1}{i}. \end{displaymath}
\end{itemize}
A proper counter assumption would, in fact, allow for different balls $B(p_{i},ir_{i})$ for every $i \in \N$, but since all the assumptions and conclusions are translation and scaling invariant, one may reduce to the case $p_{i} \equiv 0$ and $r_{0} \equiv 1$. Replacing the sets $E_{i}$ by a subsequence, one may further assume that they converge locally in the Hausdorff metric to a closed $3$-regular set $E \subset \He$. By local convergence, I mean that the Hausdorff distance between 
\begin{displaymath} E_{i} \cap \overline{B(0,N)} \quad \text{and} \quad E \cap \overline{B(0,N)} \end{displaymath}
tends to zero for every $N \in \N$ fixed: it is well-known that $3$-regularity is preserved under such convergence. It is clear that $0 \in E$, so $E \neq \emptyset$. Also, the vertical $\beta$-numbers are stable under Hausdorff convergence:
\begin{displaymath} \beta_{E}(0,1) \geq \epsilon. \end{displaymath}
Now, I claim that $E$ is symmetric, which will immediately contradict Theorem \ref{sym3reg}. Pick distinct points $p,q \in E$, and find sequences $(p_{i}),(q_{i})$ with $p_{i},q_{i} \in E_{i}$, $p_{i} \to p$ and $q_{i} \to q$. Evidently $p_{i},q_{i} \subset E_{i} \cap B(0,i)$ for $i \in \N$ sufficiently large. Consequently, for these $i$, by definition of $E_{i}$, there exist points $p_{i}' \in E_{i} \cap B(p_{i},1/i)$ and $w_{i} \in E_{i}$ such that
\begin{displaymath} d(\Sigma_{p_{i}'}(q_{i}),w_{i}) \leq \frac{1}{i}. \end{displaymath} 
Clearly $\Sigma_{p_{i}'}(q_{i}) \to \Sigma_{p}(q)$ as $i \to \infty$, so also $d(\Sigma_{p}(q),w_{i}) \to 0$, and finally $\Sigma_{p}(q) \in E$, since $E$ is closed. This shows that $E$ is symmetric and completes the proof of Lemma \ref{betasAndSymmetry}. \end{proof} 

\begin{proof}[Proof of Proposition \ref{LSImpliesWGL}] Fix $p_{0} \in E$ and $R > 0$, and consider a point $p \in B(p_{0},R)$ and a radius $0 < r \leq R$ such that $\beta_{E}(p,r) \geq \epsilon > 0$. Then, Lemma \ref{betasAndSymmetry} says that $E$ is not $\tau^{2}$-symmetric in $B(p,r/\tau)$, for some $\tau$ depending only on $\epsilon$ and the $3$-regularity constant of $E$. Consequently,
\begin{align*} \int_{0}^{R} & \calH^{3}(\{p \in B(p_{0},R) : \beta_{E}(p,r) \geq \epsilon\}) \, \frac{dr}{r}\\
& \leq \int_{0}^{R} \calH^{3}(\{p \in B(p_{0},R) : E \text{ is not $\tau^{2}$-symmetric in } B(p,r/\tau)\}) \, \frac{dr}{r} \lesssim R^{3} \end{align*} 
by \eqref{LS}. This concludes the proof of Proposition \ref{LSImpliesWGL}.
\end{proof}

\section{Boundedness of singular integrals implies local symmetry}\label{SIOSection}

In this section, I show that if all singular integrals with admissible kernels (Definition \ref{admissibleKernels}) are $L^{2}$-bounded on a closed $3$-regular set $E \subset \He$, then $E$ satisfies the local symmetry condition -- and hence the weak geometric lemma for vertical $\beta$-numbers by Proposition \ref{LSImpliesWGL}. This concludes the proof of Theorem \ref{main}.

The remaining proofs in the paper are extremely similar to arguments in \cite[Sections 2-4]{DS1}, and they are only included for the reader's convenience. I start by making a few relevant definitions; then I recall the main steps of the proof in \cite{DS1}, and finally I give a few details for the parts which are slightly different in $\He$ and $\R^{n}$.  
\begin{definition}[Admissible kernels]\label{admissibleKernels} A smooth function $K \colon \He \setminus \{0\} \to \R$ is called an \emph{admissible kernel}, if the following requirements are met:
\begin{itemize} 
\item $K$ is horizontally antisymmetric, that is $K(\bar{p}) = -K(p)$ for $p \in \He \setminus \{0\}$,
\item $K$ satisfies
\begin{displaymath} |\nabla_{\He}^{j}K(p)| \leq C(j)\|p\|^{-3 - j}, \qquad j \in \{0,1,2,\ldots\}. \end{displaymath}
\end{itemize}
\end{definition}
The functions $K_{1}(p) = X\|p\|^{-2}$ and $K_{2}(p) = Y\|p\|^{-2}$ are the primary examples of admissible kernels, see the explicit formula above \cite[Definition 2.5]{CFO2} for the horizontal antisymmetry, and \cite[Proposition 3.11]{CM} for the derivative estimate. To an admissible kernel $K$, and a number $\epsilon > 0$, one associates an operator $T_{K,\epsilon}$,
\begin{displaymath} T_{K,\epsilon}\nu(p) := \int_{\{\|q^{-1} \cdot p\| > \epsilon\}} K(q^{-1} \cdot p) \, d\nu(q), \end{displaymath} 
acting on complex Borel measures $\nu$ with finite total variation. Given a positive, locally finite Borel measure $\mu$, one says that \emph{$T_{K}$ is bounded on $L^{2}(\mu)$}, if
\begin{displaymath} \|T_{K,\epsilon}(f\mu)\|_{L^{2}(\mu)} \leq C\|f\|_{L^{2}(\mu)}, \qquad f \in L^{1}(\mu) \cap L^{2}(\mu), \end{displaymath}
for some constant $C \geq 1$ independent of $\epsilon > 0$. So, as is standard in the field, one altogether omits discussing the existence of the operator $T_{K}$. Thus, speaking about its $L^{2}$-boundedness is just a short way of expressing that the operators $T_{K,\epsilon}$ are bounded \emph{uniformly} on $L^{2}(\mu)$. 

Here is the main result of the section:

\begin{thm}\label{C1ImpliesWGL} Assume that $E \subset \He$ is a closed $3$-regular set, and $T_{K}$ is bounded on $L^{2}(\calH^{3}|_{E})$ for all admissible kernels $K$. Then $E$ satisfies the local symmetry condition, Definition \ref{LSC}. 
\end{thm}

\subsection{Steps of the proof} As I mentioned earlier, the proof of Theorem \ref{C1ImpliesWGL} follows extremely closely the argument of David and Semmes in \cite{DS1}. I will briefly explain the two main steps involved. 

In \cite{DS1}, David and Semmes study singular integrals associated to smooth, odd kernels $K \colon \R^{n} \setminus \{0\} \to \R$, satisfying the decay requirements from Definition \ref{admissibleKernels} for the Euclidean derivatives, and with "$3$" replaced by any integer $0 < m < n$ (the dimension of $E \subset \R^{n}$). For a fixed $m$-regular set $E \subset \R^{n}$, Condition "(C1)" in \cite{DS1} postulates that the operators $T_{K}$ associated to all such kernels $K$ are bounded on $L^{2}(\calH^{m}|_{E})$, in the same sense as above. According to \cite{DS1}, this is one possible definition for the uniform $m$-rectifiability of $E$.

In \cite[Section 3]{DS1}, David and Semmes show that condition "(C1)" implies another condition, known simply as "(C2)", which postulates the following (again for a fixed $m$-regular set $E \subset \R^{n}$): whenever $\psi \colon \R^{n} \to \R$ is smooth, odd, and has compact support, then
\begin{equation}\label{ConditionC2R} \sum_{2^{-k} \leq R} \int_{B(x_{0},R)} \left| \int_{E} \psi_{2^{-k}}(x - y) \, d\calH^{m}(y) \right|^{2} \, d\calH^{m}(x) \leq CR^{m}, \qquad x_{0} \in E, \: R > 0. \end{equation}
Here $\psi_{r}(z) := r^{-m}\psi(z/r)$ for $z \in \R^{n}$. 

Finding the Heisenberg analogue of the condition "(C2)" in $\He$ requires no imagination:
\begin{definition}[Condition (C2)] A closed $3$-regular set $E \subset \He$ satisfies condition (C2) if for all smooth, horizontally antisymmetric functions $\psi \colon \He \to \R$, with compact support $\spt \psi \subset \He \setminus \{0\}$, one has
\begin{equation}\label{ConditionC2} \sum_{2^{-k} \leq R} \int_{B(p_{0},R)} \left| \int_{E} \psi_{2^{-k}}(q^{-1} \cdot p) \, d\calH^{3}(q) \right|^{2} \, d\calH^{3}(p) \leq CR^{3}, \qquad p_{0} \in E, \: R > 0. \end{equation}
Here $\psi_{r}(q) = r^{-3}\psi(\delta_{r^{-1}}(q))$ for $q \in \He$.
\end{definition}

One can follow the argument in \cite[Section 3]{DS1} to arrive at the following conclusion:
\begin{proposition}\label{C1ImpliesC2} Assume that $E \subset \He$ is closed and $3$-regular, and all singular integrals associated to admissible kernels are bounded on $L^{2}(\calH^{3}|_{E})$. This is our assumption \textup{(C1)}. Then $E$ satisfies condition \textup{(C2)}.
\end{proposition}

Indeed, David and Semmes start with an odd function $\psi$, as in "(C2)". Then, to verify \eqref{ConditionC2R}, they consider antisymmetric kernels $K$ of the form
\begin{equation}\label{kernel} K(x) = \sum_{k = -N}^{N} \epsilon_{j} \cdot \psi_{2^{-k}}(x), \qquad x \in \R^{n}, \: N \in \N, \end{equation}
to which the boundedness assumption "(C1)" can be applied; here $\epsilon_{j} \in \{-1,1\}$. In the present setting, one rather starts with a horizontally antisymmetric function $\psi$, with compact support $\spt \psi \subset \He \setminus \{0\}$. Then \eqref{kernel} yields a horizontally antisymmetric kernel. The admissibility condition (ii) for $K$ follows from the assumption on $\spt \psi$: it implies that the terms in the sum defining $K(p)$ vanish for all $k$, except for those with $2^{-k} \sim_{\psi} \|p\|$. For such terms, one can estimate $|\nabla_{\He}^{j}\psi_{2^{-k}}(p)|$, $j \in \N$, by first noting that horizontal derivatives and the dilatations $\delta_{r}$ commute in the same way as Euclidean derivatives and dilatations, for example $X(\psi \circ \delta_{r})(p) = rX\psi(\delta_{r}(p))$. Consequently, the assumption (C1) about admissible kernels applies to $K$, and the proof of Proposition \ref{C1ImpliesC2} can be completed as in \cite{DS1}.

The proof of Theorem \ref{C1ImpliesWGL} has now been reduced to the claim that condition (C2) implies the local symmetry condition. Since the local symmetry condition in $\He$ is slightly different from the one employed by David and Semmes (recall Remark \ref{DScomparison}), I give all the details. It is a little questionable if this makes sense: the local symmetry condition in the present paper was slightly tweaked (compared to the original) exactly for the purpose that the following argument would work in the same way as in \cite[Section 4]{DS1}.

\begin{proof}[Proof of Theorem \ref{C1ImpliesWGL}] To show that the local symmetry condition is satisfied, one needs to fix $\tau \in (0,1)$, $p_{0} \in E$, $R > 0$, and verify the Carleson packing condition from \eqref{LS}, namely
\begin{equation}\label{form17} \int_{0}^{R} \calH^{3}(\{p \in B(p_{0},R) : E \text{ is not $\tau$-symmetric in } B(p,r)\}) \, \frac{dr}{r} \leq C_{\tau}R^{3}. \end{equation}

Before fixing a non-$\tau$-symmetric ball $B(p,r)$, I start with some preliminary constructions. Fix a constant $C \geq 1$, and let $B_{1},\ldots,B_{N}$, $N = N(\tau) \in \N$ be an enumeration of all the balls of radius $4(\tau/C)$, where the centre lies in a $(\tau/C)$-net of $\{p_{j}\}_{j \in \N} \subset B(0,5)$. Assume for a moment that the following holds for some point $q \in \He$:
\begin{equation}\label{form26} q \in E \cap B(0,5) \quad \text{and} \quad \dist(\Sigma_{0}(q),E) \geq \tau. \end{equation}
Then, if $C \geq 1$ was chosen large enough, there exists $j \in \{1,\ldots,N\}$ such that 
\begin{equation}\label{form19} B(q,\tau/C) \subset \tfrac{1}{2}B_{j} = B(p_{j},2(\tau/C)) \end{equation}
and
\begin{equation}\label{form18} \dist(\Sigma_{0}(B_{j}),E) \geq \frac{9\tau}{C}, \end{equation}
where $\Sigma_{0}(B_{j}) = \{\Sigma_{0}(p) : p \in B_{j}\}$. Simply pick a net point $p_{j}$ with $d(q,p_{j}) < \tau/C$; then the associated ball $B_{j} = B(p_{j},4(\tau/C))$ satisfies \eqref{form19}-\eqref{form18} if $C$ is large enough. Note that if $B_{j}$ satisfies \eqref{form19}-\eqref{form18}, then, using first that $q \in B_{j}$ and $\diam(B_{j}) = 8\tau/C$, and then \eqref{form18}, one gets
\begin{equation}\label{form27} \dist(\Sigma_{0}(B_{j}),B_{j}) \geq \dist(\Sigma_{0}(B_{j}),q) - \frac{8\tau}{C} \geq \frac{\tau}{C}. \end{equation}
In particular, \eqref{form27} implies that 
\begin{equation}\label{form28}\|p\| \geq \frac{\tau}{2C}, \qquad p \in B_{j} \cup \Sigma_{0}(B_{j}), \end{equation}
because otherwise $d(\Sigma_{0}(p),p) = d(\bar{p},p) \leq 2\|p\| < \tau/C$.
Only the balls $B_{j}$ arising from some $q$ as in \eqref{form26} will be of interest in the sequel, so the rest may now be discarded; in particular, one may assume that \eqref{form27}-\eqref{form28} holds for all the balls $B_{j}$ with $j \in \{1,\ldots,N\}$.

Next, for $j \in \{1,\ldots,N\}$ fixed, choose a smooth function $\tilde{\psi}^{j} \colon \He \to [0,1]$ which is supported on $\bar{B}_{j}$ and equals one on $\tfrac{1}{2}B_{j}$. Then, recalling \eqref{form27}-\eqref{form28}, and defining
\begin{displaymath} \psi(p) = \begin{cases} -\tilde{\psi}(\Sigma_{0}(p)), & p \in \Sigma(B_{j}),\\ \tilde{\psi}(p), & p \in \He \setminus \Sigma(B_{j}), \end{cases} \end{displaymath}
yields a smooth horizontally antisymmetric function $\psi$ which is non-negative outside $\Sigma_{0}(B_{j})$, and satisfies 
\begin{displaymath} \spt \psi \subset \bar{B}_{j} \cup \overline{\Sigma_{0}(B_{j})} \subset \He \setminus \{0\} \quad \text{and} \quad \psi^{j}(p) = 1 \text{ for } p \in \tfrac{1}{2}B_{j}.  \end{displaymath}

Next, moving towards \eqref{form17}, fix a ball $B$, centred at $E$ with radius $0 < r < R$, as in \eqref{form17}, where $E$ is not $\tau$-symmetric. Thus, there exist $q_{1},q_{2} \in E \cap B$ such that 
\begin{equation}\label{form24} \dist(\Sigma_{q_{1}'}(q_{2}),E) \geq \tau r \quad \text{ for all } q_{1}' \in E \cap B(q_{1},\tau r). \end{equation}
Let $k \in \Z$ be the least integer such that $2^{-k} \leq r$. Then, \eqref{form24} holds with $r$ replaced by $2^{-k}$. For notational convenience, I will assume that $r = 2^{-k}$: to be accurate, the reader should replace future occurrences of $r$ by $2^{-k}$.

Fix $q_{1}' \in E \cap B(q_{1},\tau r)$, and consider
\begin{equation}\label{form20} q := \delta_{r^{-1}}(q_{1}'^{-1} \cdot q_{2}) \quad \text{and} \quad \tilde{E} := \delta_{r^{-1}}(q_{1}'^{-1} \cdot E). \end{equation}
Then, note that $q \in \tilde{E} \cap B(0,5)$, and
\begin{displaymath} \Sigma_{0}(q) = \delta_{r^{-1}}(\overline{q_{1}'^{-1} \cdot q_{2}}) = \delta_{r^{-1}}(q_{1}'^{-1} \cdot \Sigma_{q_{1}'}(q_{2})). \end{displaymath}
By the definition of $\tilde{E}$, and \eqref{form24}, this implies that
\begin{displaymath} d(\Sigma_{0}(q),\tilde{E}) = d([\delta_{r^{-1}}(q_{1}'^{-1} \cdot \Sigma_{q_{1}'}(q_{2}))],[\delta_{r^{-1}}(q_{1}'^{-1} \cdot E)]) = \frac{\dist(\Sigma_{q_{1}'}(q_{2}),E)}{r} \geq \tau. \end{displaymath}
This means that the assumption \eqref{form26} is satisfied by the point $q$, and the set $\tilde{E}$ in place of $E$. Hence, there exists $j \in \{1,\ldots,N\}$ such that \eqref{form19}-\eqref{form18} hold (still with $\tilde{E}$ in place of $E$). Consider the associated function $\psi^{j}$. The next task will be to show that
\begin{equation}\label{form21} \int_{E} \psi^{j}_{2^{-k}}(q_{1}'^{-1} \cdot p) \, d\calH^{3}(p) = \int_{E} \psi_{r}^{j}(q_{1}'^{-1} \cdot p) \, d\calH^{3}(p) \gtrsim_{\tau} 1, \end{equation}
where, I recall, $\phi_{r}(p) = r^{-3}\phi(\delta_{r^{-1}}(p))$ for $p \in \He$. I first claim that the integrand in \eqref{form21} is non-negative for all $p \in E$. To see this, recall that $\psi^{j}$ is non-negative outside $\Sigma_{0}(B_{j})$. Then, 
\begin{displaymath} p \in E \quad \Longrightarrow \quad \delta_{r^{-1}}(q_{1}'^{-1} \cdot p) \in \tilde{E} \quad \stackrel{\eqref{form18}}{\Longrightarrow} \quad \delta_{r^{-1}}(q_{1}'^{-1} \cdot p) \notin \Sigma_{0}(B_{j}), \end{displaymath}
and this proves the claim by the definition of $\psi_{r}^{j}$. So, to prove \eqref{form18}, it suffices to show, by the $3$-regularity of $E$, that 
\begin{equation}\label{form22} \psi_{r}^{j}(q_{1}'^{-1} \cdot p) = r^{-3}, \qquad p \in B(q_{2},r\tau/C). \end{equation}
By \eqref{form19}, and the definition of $\psi^{j}$, one knows that $\psi^{j}(p) = 1$ as long as $d(p,q) \leq \tau/C$, so \eqref{form22} follows if one manages to check that
\begin{displaymath} d(\delta_{r^{-1}}(q_{1}'^{-1} \cdot p),q) \leq \frac{\tau}{C}, \qquad p \in B(q_{2},r\tau/C). \end{displaymath}
But this follows immediately from the definition of $q$ from \eqref{form20}:
\begin{displaymath} d(\delta_{r^{-1}}(q_{1}'^{-1} \cdot p),q) = d(\delta_{r^{-1}}(q_{1}'^{-1} \cdot p),\delta_{r^{-1}}(q_{1}'^{-1} \cdot q_{2})) = \frac{d(p,q_{2})}{r} \leq \frac{\tau}{C}, \quad p \in B(q_{2},r\tau/C). \end{displaymath}
This proves \eqref{form21}. Different choices of $q_{1}' \in E \cap B(q_{1},\tau r)$ -- as in \eqref{form24} -- may lead to different indices $j \in \{1,\ldots,N\}$, but in any case
\begin{equation}\label{form25} \sum_{j = 1}^{N} \int_{E} \psi^{j}_{2^{-k}}(q_{1}'^{-1} \cdot p) \, d\calH^{3}(p) \gtrsim_{\tau} 1, \qquad q_{1}' \in E \cap B(q_{1},\tau r) \subset E \cap 2B. \end{equation} 
Consequently, 
\begin{displaymath} \sum_{j = 1}^{N} \int_{E \cap 2B} \left| \int_{E} \psi_{2^{-k}}^{j}(q_{1}'^{-1} \cdot p) \, d\calH^{3}(p) \right|^{2} \, d\calH^{3}(q_{1}') \gtrsim \calH^{3}(E \cap B(q_{1},\tau r)) \sim_{\tau} \calH^{3}(E \cap B). \end{displaymath}
The proof above shows that this estimate holds for all balls $B$ which are centred on $E$, have radius in the interval $[2^{-k},2^{-k + 1})$, and with the property that $E$ is not $\tau$-symmetric in $B$. Covering the set $\{p \in B(p_{0},R) : E \text{ is not $\tau$-symmetric in } B(p,r)\}$ by such balls (with bounded overlap), this leads to the following estimate:
\begin{align*} \int_{2^{-k}}^{2^{-k + 1}} & \calH^{3}(\{p \in B(p_{0},R) : E \text{ is not $\tau$-symmetric in $B(p,r)$}\} \, \frac{dr}{r}\\
& \lesssim_{\tau} \sum_{j = 1}^{N} \int_{E \cap B(p_{0},2R)} \left| \int_{E} \psi^{j}_{2^{-k}}(q^{-1} \cdot p) \, d\calH^{3}(p) \right|^{2} \, d\calH^{3}(q). \end{align*}
Finally, summing up the intervals $[2^{-k},2^{-k + 1})$ intersecting $[0,R]$ gives
\begin{align*} \int_{0}^{R} & \calH^{3}(\{p \in B(p_{0},R) : E \text{ is not $\tau$-symmetric in } B(p,r)\}) \, \frac{dr}{r}\\
& \lesssim_{\tau} \sum_{j = 1}^{N} \sum_{2^{-k} \leq 2R} \int_{B(p_{0},2R)} \left| \int_{E} \psi_{2^{-k}}^{j}(q^{-1} \cdot p) \, d\calH^{3}(p) \right|^{2} \, d\calH^{3}(q) \, \frac{dr}{r} \lesssim R^{3}. \end{align*}
The last estimate, of course, uses the assumption (C2). The proof of of Theorem \ref{C1ImpliesWGL} is complete.
\end{proof}

Now the main result, Theorem \ref{main}, follows immediately:

\begin{proof}[Proof of Theorem \ref{main}] Combine Theorem \ref{C1ImpliesWGL} and Proposition \ref{LSImpliesWGL}. \end{proof}

\bibliographystyle{plain}
\bibliography{references}

\end{document}